\documentclass[11pt,reqno,a4paper]{amsart}
 \usepackage{amsgen, amstext,amsbsy,amsopn, amsthm, amsfonts,amssymb,amscd,amsmath,euscript,enumerate,url,verbatim,calc,tikz}
\usepackage{hyperref}
\usepackage{MnSymbol}
\usepackage{pgfkeys}
\usepackage{mathtools}
\usepackage{etex}
 \RequirePackage{etex}
\usepackage[left=1.2in,right=1.2in,bottom=1.5in]{geometry}

\usepackage{thmtools, thm-restate}
\usepackage{pst-node}

\usepackage{mathrsfs}
\usepackage{tikz-cd}

\usepackage{tikz}
\usepackage{tikz-network}
\usepackage{eqnarray,amsmath}

\def\multiset#1#2{\ensuremath{\left(\kern-.3em\left(\genfrac{}{}{0pt}{}{#1}{#2}\right)\kern-.3em\right)}}

\usetikzlibrary{arrows}

 \usepackage{latexsym}
 \usepackage{graphics}
 \usepackage{color}
\usepackage{lastpage}
\usepackage{fancyhdr}
\usepackage{multirow}
\allowdisplaybreaks
\usepackage{graphicx}
\graphicspath{ {F:/IMAGES/} }

\makeatletter
\def\oversortoftilde#1{\mathop{\vbox{\m@th\ialign{##\crcr\noalign{\kern3\p@}%
      \sortoftildefill\crcr\noalign{\kern3\p@\nointerlineskip}%
      $\hfil\displaystyle{#1}\hfil$\crcr}}}\limits}

\def\sortoftildefill{$\m@th \setbox\z@\hbox{$\braceld$}%
  \braceld\leaders\vrule \@height\ht\z@ \@depth\z@\hfill\braceru$}

\makeatother

 \newcommand{\m}{\mathfrak{m} }

  \newcommand{\Ass}{\operatorname{Ass}}

  \newcommand{\reg}{\operatorname{reg}}

 \newcommand{\depth}{\operatorname{depth}}

\newcommand{\proset}{\,\mathrel{\lower 4pt\hbox{$\scriptscriptstyle/$}
\mkern -14mu\subseteq }\,} 

 \newtheorem{theorem}{Theorem}[section]
 \newtheorem{corollary}[theorem]{Corollary}
 \newtheorem{lemma}[theorem]{Lemma}
 \newtheorem{proposition}[theorem]{Proposition}

\usepackage{amsmath}

 \theoremstyle{definition}
 
 \newtheorem{remark}[theorem]{Remark}
 \newtheorem{definition}[theorem]{Definition}

 \newtheorem{example}[theorem]{Example}

\newcommand{\Mod}[1]{\ (\mathrm{mod}\ #1)}

\newtheorem{innercustomthm}{Theorem}
\newenvironment{mythm}[1]
  {\renewcommand\theinnercustomthm{#1}\innercustomthm}
  {\endinnercustomthm}

\title{A study of $v$-number for some monomial ideals}
\author{Prativa Biswas and Mousumi Mandal }

\thanks{AMS Classification 2010: 05E40, 13F20, 13F55,05C38, 05C69}
\address{Department of Mathematics, Indian Institute of Technology Kharagpur, 721302, India}\email{prativabiswassnts@kgpian.iitkgp.ac.in}
\address{Department of Mathematics, Indian Institute of Technology Kharagpur, 721302, India}\email{mousumi@maths.iitkgp.ac.in}
\begin{document}

\maketitle
\begin{abstract}
In this paper, we give formulas for $v$-number of edge ideals of some graphs like path, cycle, 1-clique sum of a path and a cycle, 1-clique sum of two cycles and join of two graphs.
    For an $\mathfrak{m}$-primary monomial ideal $I\subset S=K[x_1,\ldots,x_t]$, we provide an explicit expression of $v$-number of $I$, denoted by $v(I)$, and give an upper bound of $v(I)$ in terms of the degree of its generators. We show that for a monomial ideal $I$, $v(I^{n+1})$ is bounded above by a linear polynomial for large $n$ and for certain classes of monomial ideals, the upper bound is achieved  for all $n\geq 1$.  For $\mathfrak m$-primary monomial ideal $I$ we prove  that $v(I)\leq \reg(S/I)$ and their difference can be arbitrarily large. 
\end{abstract}
    
\section{Introduction}
Let $S=K[x_1,x_2,\ldots,x_t]=\displaystyle\bigoplus_{k\geq 0}S_k$ be a polynomial ring in $t$ variables over a field $K$ with standard gradation. Let $I$ be a proper graded ideal of $S$. A prime ideal $\mathcal{P}$ of $S$ is an associated prime of $S/I$ if $(I:f)=\mathcal P$ for some $f\in S_k$. The set of associated primes of $S/I$ is denoted by $\Ass(S/I)$ or $\Ass(I)$. The minimal elements of $\Ass(I)$ with respect to inclusion are called minimal primes of $I$ and the rest of the elements are called embedded primes of $I$. A new invariant was introduced by Cooper \textit{et al.} in \cite{Cooper} to study the asymptotic behaviour of the minimum distance of projective Reed Muller type codes, known as $v$-number of $I$, denoted by $v(I)$, which is defined as follows:
\begin{center}\normalfont
        $v(I):=\min\{k\geq 0 ~|~ \exists f\in S_k$ and  $\mathcal{P}\in \Ass(I) ~\mbox{with}~ (I:f)=\mathcal{P}\}$.
    \end{center}
    If $(I:f)=\mathcal{P}\in \Ass(I)$ be such that $\deg f= v(I)$, then $f$ is said to be a corresponding monomial  of $v(I)$.
  For each $\mathcal P\in \Ass(I)$ the $v$-number of $I$  at $\mathcal{P}$, denoted by $v_\mathcal{P}(I)$, is defined as
  \begin{center}
      $v_\mathcal{P}(I):=\min\{k\geq 0 ~|~ \exists f\in S_k ~\mbox{with} ~(I:f)=\mathcal{P}\}$.
  \end{center}
It is clear from the definition that $v(I)\leq v_{\mathcal P}(I)$. In \cite{jaramillo} Jaramillo \textit{et al.} have studied $v(I)$ for square free monomial ideal by connecting it with the edge ideal of a clutter and have given a combinatorial formula for $v(I)$. Using this combinatorial characterization, in this paper, we give explicit formula for the $v$-number of edge ideal of a path and cycle in terms of its length.  In \cite{Saha} Saha and Sengupta have proved that if $I_1$ and $I_2$ are variable disjoint monomial ideals of $S$ then $v$-number is additive i.e. $v(I_1+I_2)=v(I_1)+v(I_2)$. We observe that additivity of $v$-number doesn't hold if $I_1$ and $I_2$ are not variable disjoint (see example \ref{exam1} and example \ref{exam2}). In this direction we prove that $v(I(H))=v(I(C_n))+v(I(C_m))-1$, where $H$ is the 1-clique sum of two cycles $C_n$ and $C_m$ of lenths $n$ and $m$ respectively. In \cite{antonino} Ficarra and Sgroi have given an explicit formula for $v(I)$, where $I$ is a monomial ideal in a polynomial ring with two variables. In this paper, we give explicit formula for $v(I)$, where $I$ is an $\m $-primary monomial ideal of $S$. In the study of $v$-number several attempts were made to compare $v(I)$ and $\reg(S/I)$. In \cite{jaramillo} and \cite{Saha}, the authors have proved that $v(I)\leq \reg(S/I)+1$ is satisfied for several classes of square free monomial ideals. However in \cite{Saha}, authors have given an example of a graph such that $v(I(G))>\reg (S/I(G))+1$. In the same paper, the authors have conjectured that for a connected graph $G$, $v(I(G)\leq \reg(S/I(G))+1$. In \cite{Civan} Civan has disproved the conjecture by showing that $v(I(G))$ can be arbitrarily larger that $\reg(S/I(G)$. In this paper we prove that for $\m$-primary monomial ideals $I$, $v(I)\leq \reg(S/I)$. Kodiyalam in \cite{Vijay} and Cutkosky, Herzog and Trung in \cite{Trung} independently have proved that $\reg (S/I^n)$ is a linear function in $n$ for $n\gg0$. Recently in \cite{antonino}, the authors have studied the asymptotic behaviour of $v$-number of powers of graded ideal of $S$. They have proved that $v(I^n)$ is bounded above and below by linear functions in $n$ for $n\gg0$. They have conjectured that if $I \subset S $ is a graded ideal with linear powers then $v(I^n)=\alpha(I)n-1$ for all $n\geq 1$. In the same paper, the authors have given affirmative answer to the conjecture for edge ideals with linear resolution, polymatriodal ideal and Hibi ideals. In this paper, independently we show that for a monomial ideal $I$, $v(I^n)$ is bounded above by a linear polynomial for $n\gg0$ with an explicit description of the linear polynomial. Further we prove that $v(I^n)$ attains the upper bound for some class of $\m$-primary monomial ideals and edge ideals.

The paper is organized as follows. In section 2, we recall some basic definitions and concepts of commutative algebra and combinatorics. Then in the following section, we give some explicit formula for $v$-number of path in terms of its number of vertices as follows:
 
\begin{mythm}{3.1}
     Let $P_n$ be a path with $n$ vertices. Then 
    \begin{center}
$v(P_n)=
\begin{cases}
   [\frac{n}{4}] & \text {if $n \equiv 0,1\Mod 4$ }
     \\ [\frac{n}{4}]+1 & \text {if $n \equiv 2, 3\Mod 4$}
\end{cases}$
\end{center}
\end{mythm}
By using this formula, we have also proved that $v(C_n)=1+v(P_{n-3})$ for $n\geq 5$ in Proposition \ref{c}. Then we give the $v$-number of 1-clique sum of path $P_m$ and cycle $C_n$ and 1-clique sum of two cycles $C_m$ and $C_n$. Also we have showed that $v(D_1*D_2)=\min\{v(D_1),v(D_2)\}$, where $D_1*D_2$ denotes the join of two graphs $D_1$ and $D_2$, in Theorem \ref{jOIN}. In Section 4, we give explicit form of $v$-number for $\mathfrak m$-primary monomial ideals in Theorem \ref{4.1}.
As a Corollary of this, we derive a part of result of Ficarra and Sgroi about $v$-number of any monomial ideals over two variables.
In \cite{antonino}, authors have provided a lower bound of $v$-number of monomial ideal by using its degree of generators. We give some upper bound for $\mathfrak m$-primary ideal in Proposition \ref{f}. We also prove that for any $\mathfrak m$-primary ideal $v(I^{s+1})\geq v(I)$ for large $s$. Saha and Sengupta in \cite[Proposition 3.11]{Saha} proved that for $\mathbf{X^G}\notin I^{n}$, $v(I^{n})\leq v(I^n:\mathbf{X}^\mathbf{G})+\deg \mathbf{X}^\mathbf{G} ,n\geq 1$, and we observe in Proposition \ref{4.6} that it is equality if $\mathbf{X}^\mathbf{G}$ is a proper divisor of corresponding monomial of $v(I^n)$. In \cite{antonino}, authors showed that $v(I^{n+1})$ is bounded above and below by linear functions and $v(I^{n+1})\leq v(I^n)+\omega(I)$ for $n\gg 0$, where $I$ is a graded ideal and $\omega(I)=\max\{d:{(I/\mathfrak{m}I)}_d\neq 0\}$. For a monomial ideal $I$, we give an explicit expression of a linear upper bound for $v(I^n)$ for large $n$ as follows:

\begin{mythm}{4.7}
     Let $I\subset S$ be a monomial ideal. Then there exists some positive integer $n_0$ such that $v(I^{n+1})\leq n\alpha(I) +d$ for all $n\geq n_0$, where $d$ is some positive integer.
\end{mythm}
We also have shown (example \ref{cycle5}) that $v(I(C_5)^{n+1})\leq 2n+1$ for $n\geq 1$, which is better bound than \cite[Proposition 2.9]{antonino}.
In Proposition 4.8 and 4.10, we get that for some special cases, the upper bound is achieved. At the end, we show that for $\mathfrak m$-primary monomial ideal $I$, $v(I)\leq\reg(S/I)$ and their difference can be arbitrarily large.

\section{Preliminaries}
Here in this section, we give some fundamental definitions, results, and notations which will be required throughout this paper.
\\In $S$, an element is said to be monomial if it is of the form $x_{1}^{a_1}\ldots x_{t}^{a_t}$ where $(a_1,\ldots,a_t)\in \mathbb{Z}_{\geq 0}^{t}$ and $\mathbb{Z}_{\geq 0}$ is the set of all non-negative integers. We write $\mathbf{{X}^{a}}=x_{1}^{a_1}\ldots x_{t}^{a_t}$, where $\mathbf{a}=(a_1,\ldots,a_t)$.
An ideal $I\subset S$ is said to be a monomial ideal if it is minimally generated by a set of monomials in $S$. For a monomial ideal, we denote the set of unique minimal generators of $I$ by $\mathscr G(I)$. For a set $A\subseteq S$, $E(A)\subseteq \mathbb{Z}_{\geq 0}^{t}$ stands for the set of all exponent vectors of elements of $A$.

\begin{definition}
    Let $I,J\subset S$ be two ideals. Then $(I:J)=\{f\in S~|~fg\in I~\mbox{for all}~g\in J\}$. For $f\in S$, $(I:f)=(I:\langle f\rangle).$ By \cite[ Proposition 1.2.2]{Hibi} for monomial ideal $I$ and monomial $f\in S$, 
    $$(I:f)=\left\langle \displaystyle\frac{u}{gcd(u,f)}~|~u\in\mathscr{G}(I)\right\rangle.$$
\end{definition}

\begin{definition}
    A clutter $\mathcal{G}$ is defined by a pair $(V(\mathcal{G}), E(\mathcal{G}))$ where $V(\mathcal{G})$ denotes the set of all vertices of $\mathcal{G}$ and  $E(\mathcal{G})$, called edges, is the family of subsets of $V(\mathcal{G})$ such that they are pairwise incomparable with respect to inclusion. A simple graph $\mathcal{G}$ is a clutter such that the cardinality of each element of $E(\mathcal{G})$ is two. 
\end{definition}

Suppose $\mathcal{G}$ is a clutter with $V(\mathcal{G})=\{x_1,x_2,\ldots,x_t\}$ and $E(\mathcal{G})$ as edge set. We observe each vertex $x_i$ as a variable of the polynomial ring $S=K[x_1,x_2,\ldots,x_t]$ in $t$ variables over a fixed field $K$ and by defining the following ideal we can relate each clutter to an ideal.
The edge ideal associated with clutter $\mathcal{G}$ is the monomial ideal
\begin{center}
    $I(\mathcal{G})=\langle\prod_{x_i\in e}x_i~|~e\in E(\mathcal{G})\rangle.$
\end{center}
It is a square-free ideal. Hence we can get $v$-number of edge ideal. 

    


Next, we recall Some definitions of special classes of graphs that we have considered in this paper:
\begin{definition}
    A simple graph is said to be a path if its vertices can be ordered such that two vertices are adjacent if and only if they are consecutive in the list. If a path has $n$ vertices, then it is denoted by $P_n$.
\end{definition}
\begin{definition}
    A cycle is a simple graph with the same number of vertices and edges whose vertices can be placed around a circle so that two vertices are adjacent if and only if they appear consecutively along the circle. A cycle with $n$ vertices is denoted by $C_n$.
\end{definition}
\begin{definition}
    Let $G_1$ and $G_2$ be two simple graphs. Suppose that $G_1\cap G_2=K_r$ is a complete graph of order $r$, where $G_1\neq K_r$ and $G_2\neq K_r$. Then $H=G_1\cup G_2$ is said to be $r$-clique sum of $G_1$ and $G_2$.
\end{definition}

\begin{definition}\normalfont\label{joingraph}
    Let $r\geq 2$ be an integer and $G_1,\ldots,G_r$ be simple graphs with pairwise disjoint vertex sets. Then join of $G_1,\ldots,G_r,$ denoted by $G_1*\cdots*G_r$, is a simple graph over the vertex set $V(G_1)\sqcup\ldots\sqcup V(G_r)$ with edge set $E(G_1)\cup \ldots\cup E(G_r)\cup \{(x,y): x\in V(G_i),y\in V(G_j)~\mbox{with} ~i< j~\}$.
\end{definition}
Let us recall some invariants from graph theory.
\begin{definition}

    A subset $C\subset V(\mathcal{G})$ is said to be a vertex cover of a graph $\mathcal{G}$ if for any $e\in E(\mathcal{G})$, $e\cap C\neq\phi$. If vertex cover is minimal with respect to inclusion, then it is a minimal vertex cover.
\end{definition}
\begin{definition}
    A subset $A\subset V(\mathcal{G})$ is stable or independent if $e\not\subseteq A$ for all $e\in E(\mathcal{G})$. A maximal independent set is an independent set, which is maximal with respect to inclusion. 
\end{definition}

    For a stable set $A$ in $\mathcal{G}$, the neighbor set of $A$ in $\mathcal{G}$, denoted by $N_{\mathcal{G}}(A)$, is defined by
    $$N_{\mathcal{G}}(A)=\{x_i\in V(\mathcal{G})~|~\{x_i\}\cup A ~\mbox{contains}~ \text{an edge}~ of~ \mathcal{G}\}.$$
We write $N_{\mathcal{G}}[A]:=N_{\mathcal{G}}(A)\cup A$ and let
    $\mathcal{A}_\mathcal{G}$ be the collection of those stable sets $A$ of $\mathcal{G}$ such that neighbour set of $A$ is minimal vertex cover of $\mathcal{G}$.
The next Theorem gives the combinatorial description of $v$-number.
\begin{theorem}\label{graph}\cite[Theorem 3.5]{jaramillo}
    Let $I$ be the edge ideal of a clutter $\mathcal{G}$. Then $v(I)=\min\{|A|:A\in \mathcal{A}_\mathcal{G}\}$.
\end{theorem}
Now onwards $v$-number of edge ideal of a graph $G$ is denoted by $v$-number of $G$ and we write $v(G):=v(I( G))$.
Let $A\in \mathcal{A}_{G}$ be such that $v(I)=|A|$, then $A$ is said to be corresponding stable set of $v$-number of ${G}$.
\\For a non-zero graded ideal $I$, define $\alpha (I):=\min\{\deg (f)~|~f\in I\backslash \{0\}\}$. We get a lower bound for $v$-number for any monomial in terms of the degree of its generator from the following statement.
\begin{proposition}\normalfont\label{e}
(\cite[Proposition 4.3.]{antonino}) Let $I\subset S$ be a monomial ideal. Then $v_\mathcal{P}(I)\geq \alpha(I)-1$ for all $\mathcal P\in \Ass(I)$.

\end{proposition}
The following statement is about upper bound of $v$-number of $I$.
\begin{proposition}\normalfont\label{Saha}
    \cite[Proposition 3.11]{Saha}
    Let $I$ be a monomial ideal and $f$ be a monomial such that $f\notin I$. Then $v(I)\leq v(I:f)+\deg f$.
\end{proposition}

Next, we will define another important invariant of commutative algebra.
    Let $M$ be a finitely generated $S$ module. Then we can write the graded minimal free
resolution of $M$ in the following form
    $$0\rightarrow\displaystyle\bigoplus_{j} S(-j)^{\beta_{g,j}}\rightarrow\cdot\cdot\cdot\rightarrow\displaystyle\bigoplus_j S(-j)^{\beta_{1,j}}\rightarrow \displaystyle\bigoplus S(-j)^{\beta_{0,j}}\rightarrow M\rightarrow 0.$$
    Here $\beta_{ij}$ is the $(i,j)$-th graded Betti number of $M$ and $S(-j)$ denote the polynomial ring shifted in degree $j$.
\begin{definition}
    The Castlenuovo-Mumford regularity (or regularity in short) of $M$, denoted as $\reg(M)$, is defined by 
    $$\reg(M)=\max\{j-i~|~\beta_{i,j}\neq 0\}.$$
\end{definition}
The regularity can also be defined via the vanishings of local cohomology modules with respect to the unique maximal homogeneous ideal
$\mathfrak m=\langle x_1,x_2,\ldots,x_t\rangle$. Let us define for $i\geq 0$,
\begin{center}
    $a_i(M):=
    \begin{cases}
        \max\{j~|~H_{\mathfrak m}^{i}(M)_j\neq0\} & \text {if $H_{\mathfrak m}^{i}(M)\neq 0$}
        \\-\infty & \text{otherwise}
    \end{cases}$
\end{center}
Then $\reg (M)=\max\{a_i(M)+i~|~i\geq 0\}$.
\begin{center}
\section{\texorpdfstring{$v$}a-number of some of graphs}
\end{center}
In this section, we give an explicit formula for the $v$-number of a path $P_n$ in terms of its number of vertices. We further show the relation between the $v$-number of a path $P_n$ and that of a cycle $C_n$. We also discuss the $v$-numbers of graphs like clique sum and join of graphs.

\begin{theorem}
\label{a}
    Let $P_n$ be a path with $n$ vertices. Then
    \begin{center}
$v(P_n)=
\begin{cases}
   [\frac{n}{4}] & \text {if $n \equiv 0,1\Mod 4$ }
     \\ [\frac{n}{4}]+1 & \text {if $n \equiv 2, 3\Mod 4$}
\end{cases}$
\end{center}
\end{theorem}
\begin{proof}
    We prove the Theorem by dividing it into the following four cases:
  \\ Case-1: Let $n=4t$ for some $t\in \mathbb N$. We proceed by induction on $t$. For $t=1$, we have $v(P_{4})=1$. A stable set corresponding to $v(P_4)$ is given by $A_4=\{x_3\}$.
   
   \vspace*{0.5 cm}
   \begin{center}

   \begin{tikzpicture}
    \Vertex [x=3,label=$x_1$,position=below,size=0.1,color=black]{A}
    \Vertex [x=4.5,label=$x_2$,position=below,size=0.1,color=black]{B}
\Edge (A)(B)
\Vertex [x=6,label=$x_3$,position=below,size=0.1,color=black]{C}
\Edge (C)(B)
\Vertex [x=7.5,label=$x_4$,position=below,size=0.1,color=black]{K}
\Edge (C)(K)
\end{tikzpicture}
 \end{center}

 Suppose the statement is true for $n=4t$ and we prove the statement for $n=4(t+1)$. By induction hypothesis $v(P_{4t})=t$ with corresponding stable set $A_{4t}=\{x_3,x_7,\ldots,x_{4t-1}\}$.

  \begin{center}

   \begin{tikzpicture}[scale=0.9]
    \Vertex [x=1,label=$x_1$,position=below,size=0.1,color=black]{A}
    \Vertex [x=2,label=$x_2$,position=below,size=0.1,color=black]{B}
\Edge (A)(B)
\Vertex [x=3,label=$x_3$,position=below,size=0.1,color=black]{C}
\Edge (C)(B)
\Vertex [x=4,label=$x_4$,position=below,size=0.1,color=black]{D}
\Edge (C)(D)
\Vertex [x=5,size=0.05,color=black]{E}

\Vertex [x=6,size=0.05,color=black]{G}

\Vertex [x=7,label=$x_{4t-1}$,position=below,size=0.1,color=black]{H}
    \Vertex [x=8,label=$x_{4t}$,position=below,size=0.1,color=black]{I}
\Edge (H)(I)

\end{tikzpicture}
\end{center}
Now suppose $n=4(t+1)$.
\vspace*{0.5 cm}
\begin{center}
\begin{tikzpicture}
    \Vertex [x=1,label=$x_1$,position=below,size=0.1,color=black]{A}
    \Vertex [x=2,label=$x_2$,position=below,size=0.1,color=black]{B}
\Edge (A)(B)
\Vertex [x=3,label=$x_3$,position=below,size=0.1,color=black]{C}
\Edge (C)(B)
\Vertex [x=4,label=$x_4$,position=below,size=0.1,color=black]{D}
\Edge (C)(D)
\Vertex [x=4.5,size=0.05,color=black]{E}

\Vertex [x=5,size=0.05,color=black]{F}
\Vertex [x=5.5,size=0.05,color=black]{G}

\Vertex [x=6,label=$x_{4t-1}$,position=below,size=0.1,color=black]{H}
    \Vertex [x=7,label=$x_{4t}$,position=below,size=0.1,color=black]{I}
\Edge (H)(I)
\Vertex [x=8,label=$x_{4t+1}$,position=below,size=0.1,color=black]{J}
\Edge (I)(J)
\Vertex [x=9,label=$x_{4t+2}$,position=below,size=0.1,color=black]{K}
\Edge (J)(K)
\Vertex [x=10,label=$x_{4t+3}$,position=below,size=0.1,color=black]{L}
\Edge (L)(K)
\Vertex [x=11,label=$x_{4t+4}$,position=below,size=0.1,color=black]{M}
\Edge (M)(L)
\end{tikzpicture}
\end{center} 
Clearly, from the above figure,
 $N_{P_{4t+4}}(\{x_1\})=\{x_2\}$ and $N_{P_{4t+4}}(\{x_{4t+4}\})=\{x_{4t+3}\}$, and each of which covers two edges. Also,
 $N_{P_{4t+4}}(\{x_2\})=\{x_1,x_3\}$ and $N_{P_{4t+4}}(\{x_{4t+3}\})=\{x_{4t+2},x_{4t+4}\}$ and each of which covers three edges.
For any $~ 3\leq i \leq 4t+2$, $N_{P_{4t+4}}(\{x_i\})=\{x_{i-1},x_{i+1}\}$ which covers exactly four edges. 
Let $B$ be a stable set of $P_{4(t+1)}$ with cardinality $s$, such that $N_{P_{4t+4}}(B)$ is a vertex cover of $P_{4t+4}$. For $1\leq i\leq 4t+4$, $N_{P_{4t+4}}(\{x_i\})$ can cover at most $4$ edges, then $4s\geq 4t+3$, which implies $ s\geq (t+1) >t$
. Hence $v(P_{4t+4})> t.$
Now $A_{4t}\cup\{x_{4t+3}\}$ is a stable set and $N_{P_{4t+4}}(A_{4t}\cup\{x_{4t+3}\})$ is vertex cover of $P_{4t+4}.$ Therefore, $v(P_{4t+4})\leq (t+1).$ Thus $v(P_{4t+4})=t+1$ with stable set $A_{4t+4}=\{x_3,x_7,\ldots,x_{4t-1},x_{4t+3}\}$. 
\\By similar arguments, we have the following:
\\Case-2: Let $n=4t+1$, then $v(P_{n})=[\frac{n}{4}]$  with corresponding stable set $A_n=\{x_3,x_7,\ldots,x_{4t-1}\}$. 
\\Case-3: Let $n=4t+2$, then $v(P_{n})=[\frac{n}{4}]+1$  with corresponding stable set $A_n=\{x_2,x_6,\ldots,x_{4t+2}\}$.
\\Case-4: Let $n=4t+3$, then $v(P_{n})=[\frac{n}{4}]+1$  with corresponding stable set $A_n=\{x_2,x_6,\ldots,x_{4t+2}\}$.\end{proof}
In the following remark, we give another set of stable set corresponding to $v(P_n)$.
\begin{remark}
Note that in the proof of the above Theorem the corresponding stable set is not unique.  
    For $n=4t$,  $A_n=\{x_2,x_6,\ldots,x_{4t-2}\}$ is another stable set corresponding to $v(P_n)$. Similarly, for $n=4t+1$, $n=4t+2$ and $n=4t+3$, $A_n=\{x_2,x_6,\ldots,x_{4t-2}\}$, $A_n=\{x_1,x_5,\ldots,x_{4t+1}\}$ and $A_n=\{x_1,x_5,\ldots,x_{4t+1}\}$ are stable set corresponding to $V(P_{n})$ respectively.
\end{remark}
The next Lemma will be useful in proving the relation between the $v$-number of a path and a cycle.
\begin{lemma} \label{b}
Let $P_n$ be a path with $n$ vertices. Then any stable set corresponding to $v(P_n)$ can not contain two endpoints together.
\end{lemma}
\begin{proof}
Let if possible, $A$ be a corresponding stable set of $v(P_n)$ containing two endpoints $x_1$ and $x_n$. 
   \begin{center}
   \begin{tikzpicture}[scale=0.5]
    \Vertex [x=1,label=$x_1$,position=below,size=0.1,color=black]{A}
    \Vertex [x=3,label=$x_2$,position=below,size=0.1,color=black]{B}
\Edge (A)(B)
\Vertex [x=5,label=$x_3$,position=below,size=0.1,color=black]{C}
\Edge (C)(B)
\Vertex [x=7,label=$x_4$,position=below,size=0.1,color=black]{D}
\Edge (C)(D)
\Vertex [x=8,size=0.05,color=black]{E}

\Vertex [x=9,size=0.05,color=black]{F}
\Vertex [x=10,size=0.05,color=black]{G}

\Vertex [x=11,label=$x_{n-1}$,position=below,size=0.1,color=black]{H}
    \Vertex [x=13,label=$x_{n}$,position=below,size=0.1,color=black]{I}
\Edge (H)(I)

\end{tikzpicture}
\end{center}
Note that $N_{P_n}(\{x_1\})=\{x_2\}$ and $N_{P_n}(\{x_{n}\})=\{x_{n-1}\}$ and each of which covers two edges. Then we are left with $P_{n-4}$, where $V(P_{n-4})=\{x_3,x_4,\ldots,x_{n-2}\}$.  By Theorem \ref{graph}, we have $|A|\geq 2+v(P_{n-4})$. Let $v(P_{n-4})=|A'|$, where $A'\subset \{x_3,\ldots,x_{n-2}\}$ is a corresponding stable set of $v(P_{n-4})$. Let $A''=\{x_{i+2}\in V(P_n)~|~x_i\in A'\}$, then clearly, $|A''|=|A'|=v(P_{n-4})$. Also $\{x_3\}\cup A''$ is a stable set and $N_{P_n}(A''\cup\{x_3\})$ is a vertex cover of $P_n$. Thus, $v(P_n)\leq 1+v(P_{n-4})<|A|$, which is a contradiction. Therefore, $A$ can not contain two endpoints together.
\end{proof}
\begin{proposition} \label{c}

Let $P_n$ be a path and $C_n$ be a cycle. Then for $n\geq 5$
\begin{center}

    $v(C_n)=v(P_{n-3})+1$.
    \end{center}
\end{proposition}
\begin{proof}
    Let for $n\geq 5,~V(C_n)=V(P_n)=\{x_1,x_2,\ldots,x_n\}$ and $V(P_{n-3})=\{x_1,x_2,\ldots,x_{n-3}\}$.
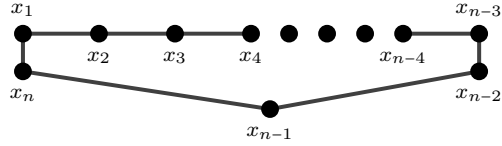
    \begin{center}
   \begin{tikzpicture}[scale=0.5]
    \Vertex [x=1,y=2,label=$x_1$,position=above,size=0.1,color=black]{A}
    \Vertex [x=3,y=2,label=$x_2$,position=below,size=0.1,color=black]{B}
\Edge (A)(B)
\Vertex [x=5,y=2,label=$x_3$,position=below,size=0.1,color=black]{C}
\Edge (C)(B)
\Vertex [x=7,y=2,label=$x_4$,position=below,size=0.1,color=black]{D}
\Edge (C)(D)
\Vertex [x=8,y=2,size=0.05,color=black]{E}

\Vertex [x=9,y=2,size=0.05,color=black]{F}
\Vertex [x=10,y=2,size=0.05,color=black]{G}

\Vertex [x=11,y=2,label=$x_{n-4}$,position=below,size=0.1,color=black]{H}
    \Vertex [x=13,y=2,label=$x_{n-3}$,position=above,size=0.1,color=black]{I}
\Edge (H)(I)

 \Vertex [x=1,y=1,label=$x_n$,position=below,size=0.1,color=black]{J}
 \Edge (A)(J)
 \Vertex [x=13,y=1,label=$x_{n-2}$,position=below,size=0.1,color=black]{K}
 \Edge (I)(K)
 \Vertex [x=7.5,y=0,label=$x_{n-1}$,position=below,size=0.1,color=black]{L}
 \Edge (L)(K)
 \Edge (J)(L)
\end{tikzpicture}
\end{center}
Suppose $v(P_{n-3})=|A|$, where $A$ is a corresponding stable set.
By Lemma \ref{b}, $A$ can not contain the two endpoints $x_1$ and $x_{n-3}$ together. Therfore, $v(P_{n-3})<v(C_n)$.
Note that $A\cup \{x_{n-1}\}$ is a stable set of $C_n$ and its neighbor set is a vertex cover of $C_n$. Therefore, $v(C_n)\leq v(P_{n-3})+1$. Thus, $v(C_n)=v(P_{n-3})+1$ for all $n\geq 5$, with a corresponding stable set  $ A\cup\{x_{n-1}\}$.
\end{proof}
From \cite[Proposition 3.9]{Saha}, we get the additivity of $v$-number holds for variable disjoint monomial ideals. If ideals are not variable disjoint, then by the following examples we see that additivity is not true.
\begin{example}\label{exam1}
    Let $I_1=\langle x_1x_2,x_2x_3,x_3x_4,x_2x_5,x_5x_6\rangle\subset K[x_1,\ldots,x_6]$ and $I_2=\langle x_1x_7,x_7x_8,\\x_1x_9,x_9x_{10}\rangle\subset K[x_1,x_7,x_8,x_9,x_{10}]$ be two square-free monomial ideals, which are edge ideal of graph $G_1$ and $G_2$. Then $$I_1+I_2 =\langle x_1x_2,x_2x_3,x_3x_4,x_2x_5,x_5x_6,x_1x_7,x_7x_8,x_1x_9,x_9x_{10}\rangle\subset K[x_1,\ldots,x_{10}],$$ which corresponds the edge ideal 1-clique sum of $G_1$ and $G_2$. Here $v(I_1)=1,v(I_2)=1$ and $v(I_1+I_2)=3$, which gives $v(I_1)+v(I_2)<v(I_1+I_2)$. 
\end{example}
\begin{example}\label{exam2}
Let $I_1=\langle x_1x_2,x_2x_3 \rangle\subset K[x_1,x_2,x_3]$ and $I_2=\langle x_3x_4,x_4x_5\rangle\subset K[x_3,x_4,x_5]$ be two square-free monomial ideals, which are edge ideals of graph $G_1$ and $G_2$. Then $$I_1+I_2=\langle x_1x_2,x_2x_3,x_3x_4,x_4x_5\rangle\subset K[x_1,\ldots,x_5],$$ which corresponds the edge ideal 1-clique sum of $G_1$ and $G_2$. Here $v(I_1)=1,v(I_2)=1$ and $v(I_1+I_2)=1$, which gives $v(I_1)+v(I_2)>v(I_1+I_2)$. 
\end{example}
From the above examples, we can say that neither of the two inequality is true in general. Next we discuss the $v$-number of $1$-clique sum of a cycle and a path.
\begin{proposition} \label{prativa}
Let H be a $1$-clique sum of a cycle $C_n$ and a path $P_m$. Then 
\begin{center}

    $v(C_n)+v(P_{m-2})-1\leq v(H)\leq v(C_n)+v(P_{m-2})$. 
\end{center}
\end{proposition}

\begin{proof}
Let $V(C_n)=\{x_1,\ldots,x_n\}$, $V(P_m)=\{x_1,t_2,\ldots,t_m\}$ and $H$ be $1$-clique sum of $C_n$ and $P_m$, joined at the vertex $x_1$. Suppose $v(C_n)=s$ and A is a corresponding stable set of $v(C_n)$ i.e, $v(C_n)=|A|$.

\begin{center}

    \begin{tikzpicture}
         \Vertex [x=2,y=3,label=$x_2$,position=above,size=0.1,color=black]{A}
          \Vertex [x=1,y=2,label=$x_3$,position=above,size=0.1,color=black]{B}
         \Edge(A)(B)

            \Vertex [x=3,y=2,label=$x_1$,position=above,size=0.1,color=black]{C}
            \Edge(A)(C)
             
             \Vertex [x=3,y=1,label=$x_n$,position=right,size=0.1,color=black]{D}
             \Edge(C)(D)
             \Vertex [x=1,y=1.5,position=left,size=0.05,color=black]{E}
             
             \Vertex [x=1,y=1,position=left,size=0.05,color=black]{F}
             \Vertex [x=1.5,y=1,position=left,size=0.05,color=black]{F}
              \Vertex [x=2,y=1,position=left,size=0.05,color=black]{F}
              \Vertex [x=2.5,y=1,position=left,size=0.05,color=black]{F}
               \Vertex [x=4,y=3,label=$t_2$,position=above,size=0.1,color=black]{G}
               \Vertex [x=5,y=2,label=$t_3$,position=above,size=0.1,color=black]{H}
           \Edge(G)(C)
           \Edge(G)(H)
           \Vertex [x=6,y=3,label=$t_4$,position=above,size=0.1,color=black]{I}
           \Edge(H)(I)
           \Vertex [x=6.5,y=3,position=above,size=0.1,color=black]{G}
            \Vertex [x=7,y=3,position=above,size=0.1,color=black]{G}
            \Vertex [x=7.5,y=3,label=$t_{m-1}$,position=above,size=0.1,color=black]{I}
            \Vertex [x=8.5,y=2,label=$t_m$,position=above,size=0.1,color=black]{J}
            \Edge(I)(J)
    \end{tikzpicture}
    \end{center}
Without the loss of generality, let $x_1\in A$.
Clearly, $t_2\in N_H(A)$, so $N_H(A)$ is a vertex cover of the induced subgraph with vertex set $V(C_n)\cup \{t_2,t_3\}$. Then we need to cover the path $P_{m-2}$, where V($P_{m-2})=\{t_3,\ldots,t_{m-1},t_m\}$. Suppose B is a stable set corresponding to $v(P_{m-2})$, then $A\cup B$ is a stable set of $H$, such that $N_H(A\cup B)$ is a vertex cover of $H$. Thus, $v(H)\leq |A\cup B|=|A|+|B|=v(C_n)+v(P_{m-2})$. 
For the other inequality, let $D$ be a stable set corresponding to $v(H)$. Then we have the following two cases:
\\Case-I: If $x_1\in D$, then clearly $N_{C_n}(D\cap V(C_n))$ and $N_{P_m}(D\cap V(P_m))$ are vertex covers of $C_n$ and $P_m$ respectively. This gives $|D\cap V(C_n)|+|D\cap V(P_m)|\geq v(C_n)+v(P_m)\geq v(C_n)+v(P_{m-2})$. Thus, $|D|+1\geq v(C_n)+v(P_{m-2})$, which implies $v(H)\geq v(C_n)+v(P_{m-2})-1$.Let $N'_{C_n}(D)=\{x\in V(C_n)~|~\{x\}\cup D ~\mbox{contains an edge of} ~H\}$ and $N'_{P_m}(D)=\{x\in V(P_m)~|~\{x\}\cup D ~\mbox{contains an edge of} ~H\}$.
\\Case-II: If $x_1\notin D$, then we have the following:

\begin{enumerate}[(i)]
    \item If $N_{C_n}(D\cap V(C_n))$ is not a vertex cover of $C_n$, then $(N'_{C_n}(D)\cap V(C_n))\backslash N_{C_n}(D\cap V(C_n))=\{x_1\}$ and $(N'_{P_m}(D)\cap V(P_m))\backslash N_{P_m}(D\cap V(P_m))=\phi$. This implies $N_{C_n}(\{D\cap V(C_{n})\}\cup\{x_1\})$ is a vertex cover of $C_n$ and $N_{P_m}(D\cap V(P_m))$ is a vertex cover of $P_m$. Thus, $|\{D\cap V(C_{n})\}\cup\{x_1\}|+|D\cap V(P_m)|\geq v(C_n)+v(P_m)\geq v(C_n)+v(P_{m-2})$. Therefore, $v(H)=|D|\geq v(C_n)+v(P_{m-2})-1$.

    \item If $N_{P_m}(D\cap V(P_m))$ is not a vertex cover of $P_m$, then $(N'_{P_m}(D)\cap V(P_m))\backslash N_{P_m}(D\cap V(P_m))=\{x_1\}$ and $(N'_{C_n}(D)\cap V(C_n))\backslash N_{C_n}(D\cap V(C_n))=\phi$. Thus, $N_{P_m}(\{D\cap V(P_m)\}\cup\{x_1\})$ is a vertex cover of $P_m$ and $N_{C_n}(D\cap V(C_n))$ is a vertex cover of $C_n$, which implies $v(H)=|D|\geq v(C_n)+v(P_{m-2})-1$.

    \item If $N_{C_n}(D\cap V(C_n))$ and $N_{P_m}(D\cap V(P_m))$ are the vertex covers of $C_n$ and $P_m$ respectively, then $|D|\geq v(C_n)+v(P_m)\geq v(C_n)+v(P_{m-2})$.
\end{enumerate} 
\end{proof}
Consequently, as a Corollary, we prove that the lower bound is achieved in the above Proposition for some specific values of $n$ and $m$.
\begin{corollary}\label{csk}[With the hypothesis as in \ref{prativa}]
    If $n\equiv 1,2\Mod 4$ and $m\equiv 0  \Mod 4$, then 
  $v(H)= v(C_n)+v(P_{m-2})-1.$   

\end{corollary}
\begin{proof}
Let $V(C_n)=\{x_1,\ldots,x_n\}$, $V(P_m)=\{x_1,t_2,\ldots,t_m\}$ and $H$ be $1$-clique sum of $C_n$ and $P_m$, joined at the vertex $x_1$. By Theorem \ref{a}, $A=\{t_2,\ldots,t_{m-2}\}$ is a stable set corresponding to $v(P_m)$ for $m\equiv 0\Mod 4$. By Proposition \ref{c} and Theorem \ref{a}, for $n\equiv 1 \Mod 4$, we have $B=\{x_1\}\cup \{x_4,\ldots,x_{n-1}\}$. Also, for $n\equiv 2 \Mod 4$, we have $B=\{x_1\}\cup\{x_4,\ldots,x_{n-2}\}$, the  stable set corresponding to $v(C_n)$. Clearly, $(A\cup B)\backslash\{x_1\}$ is a stable set and its neighbor set is a vertex cover of $H$. Therefore, $v(H)\leq |(A\cup B)\backslash\{x_1\}|=v(C_n)+v(P_m)-1=v(C_n)+v(P_{m-2})-1.$ 
\end{proof}
In the next Proposition, we give the formula for the $v$-number of $1$-clique sum of two cycles.
\begin{proposition}

    Let H be a $1$-clique sum of two cycles $C_m$ and $C_n$. Then
    \begin{center}
        $v(H)= v(C_n)+v(C_m)-1.$
    \end{center}
\end{proposition}
\begin{proof}
Let $V(C_m)=\{x_1,x_2,\ldots,x_m\}$, $V(C_n)=\{x_1,t_2,\ldots,t_n\}$ and $H$ be $1$-clique sum of $C_m$ and $C_n$, joined at the vertex $x_{1}$.
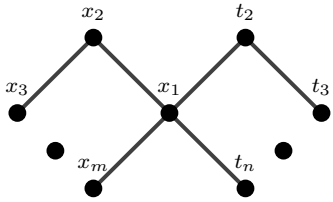
\begin{center}

    \begin{tikzpicture}
\Vertex [x=2,y=3,label=$x_2$,position=above,size=0.1,color=black]{A}
\Vertex [x=1,y=2,label=$x_3$,position=above,size=0.1,color=black]{B}
\Edge(A)(B)
\Vertex [x=3,y=2,label=$x_1$,position=above,size=0.1,color=black]{C}
\Edge(A)(C)
\Vertex [x=2,y=1,label=$x_m$,position=above,size=0.1,color=black]{D}
\Edge(D)(C)
\Vertex [x=1.5,y=1.5,position=above,size=0.1,color=black]{E}
\Vertex [x=4,y=3,label=$t_2$,position=above,size=0.1,color=black]{F}
\Edge(C)(F)
\Vertex [x=4,y=1,label=$t_n$,position=above,size=0.1,color=black]{G}
\Edge(C)(G)
\Vertex [x=5,y=2,label=$t_3$,position=above,size=0.1,color=black]{H}
\Edge(H)(F)
\Vertex [x=4.5,y=1.5,position=above,size=0.1,color=black]{B}
    \end{tikzpicture}
    \end{center}
    Suppose $A$ and $B$ are stable sets corresponding to $v(C_m)$ and $v(C_n)$ respectively with $x_1\in A\cap B$. Clearly, $A\cup B$ is a stable set whose neighborhood is a vertex cover of $H$. Then $v(H)\leq |A\cup B|=|A|+|B|-|A\cap B|=v(C_m)+v(C_n)-1$.
     Let $D$ be a stable set corresponding to $v(H)$. Now, if $x_1\in D$, then $N_{C_m}(V(C_m)\cap D)$ and $N_{C_n}(V(C_n)\cap D)$ are vertex covers of $C_m$ and $C_n$ respectively. This implies $|D|+1= |D\cap V(C_n)|+|D\cap V(C_m)|\geq v(C_n)+v(C_m)$. Therefore, $v(H)=|D|\geq v(C_n)+v(C_m)-1$. Let $N'_{C_n}(D)=\{x\in V(C_n)~|~\{x\}\cup D ~\mbox{contains an edge of} ~H\}$ and $N'_{C_m}(D)=\{x\in V(C_m)~|~\{x\}\cup D ~\mbox{contains an edge of} ~H\}$. If $x_1\notin D$, then we have the following three cases:
     \begin{enumerate}[(i)]
         \item  If $N_{C_m}(V(C_m)\cap D)$ is not a vertex cover of $C_m$, then $\{N'_{C_m}(D)\cap V(C_m)\}\backslash N_{C_m}(D\cap V(C_m))=\{x_1\}$ and $N'_{C_n}(D)\cap V(C_n)=N_{C_n}(D\cap V(C_n))$. This implies $N_{C_m}(\{D\cap V(C_m)\}\cup \{x_1\})$ and $N_{C_n}(D\cap V(C_n))$ are vertex covers of $C_m$ and $C_n$ respectively. Therefore, $|D|+1\geq v(C_m)+v(C_n)$. Since $v(H)=|D|$, we have $v(H)\geq v(C_m)+v(C_n)-1$. Hence $v(H)=v(C_m)+v(C_n)-1.$ 
         \item Similarly, if $N_{C_n}(V(C_n)\cap D)$ is not vertex cover of $C_n$, then we get $v(H)\geq v(C_m)+v(C_n)-1$. Hence $v(H)=v(C_m)+v(C_n)-1.$ 
         \item If $N_{C_m}(V(C_m)\cap D)$ and $N_{C_n}(V(C_n)\cap D)$ are vertex covers of $C_m$ and $C_n$ respectively, then $|D|\geq v(C_m)+v(C_n)$ which contradicts the fact that $v(H)\leq v(C_m)+v(C_n)-1$. Therefore this case can not occure.
     \end{enumerate}
\end{proof}

We end this section with a Theorem that relates the $v$-number of join graph with the $v$-number of its individual graphs.
\begin{theorem}\label{jOIN}

Let $D=D_1*D_2$ be the join of two graphs $D_1$ and $D_2$. Then 
\begin{center}
    $v(D) = \min\{v(D_1),v(D_2)\}$.
\end{center}
\end{theorem}
\begin{proof}
Let $v(D_1)=|A|$, where $A$ is a stable set and $N_{D_1}(A)$ is a vertex cover of $D_1$.
    By definition \ref{joingraph},  $A$ is also a stable set and 
   $N_{D_1}(A)$ is a vertex cover of $D$.
Therefore, $v(D)\leq |A|= v(D_1)$. Let $v(D_2)=|A'|$, where $A'$ is a stable and $N_{D_2}(A')$ is a vertex cover of $D_2$. Then similarly we have $v(D)\leq |A'|= v(D_2)$. Thus, $v(D)\leq \min\{v(D_1), v(D_2)\}$.
Let $v(D)=|B|$, where $B$ is a stable set and $N_D(B)$ is a vertex cover of $D$. Clearly, $B\subset V(D_1)$ or $B\subset V(D_2)$, otherwise it will not be a stable set. Without the loss of generality, let $B\subset V(D_1)$ which implies $B$ is a stable set in $D_1$ and its neighborhood is a vertex cover of $D_1$, otherwise, it will contradict that $N_D(B)$ is vertex cover of $D$. Thus, $v(D_1)\leq |B|=v(D)$ which implies min$\{v(D_1), v(D_2)\}\leq v(D).$ Therefore, $v(D) = \min\{v(D_1),v(D_2)\}$.
    \end{proof}

    \section{\texorpdfstring{$v$}a-number of monomial ideal}
    
    In this section, we explicitly give a formula for the $v$-number of an $\mathfrak{m}$-primary monomial ideal. Moreover, we prove that $v(I^{n+1})$ is bounded above by a linear polynomial and this upper bound is achieved for certain classes of graphs.

 Let $\mathcal{M}$\label{M} denote the collection of all the $t\times t$ matrices $A=(a_{ij})$, such that the rows of $A$ are exponent vectors of monomials  constructed by choosing $t$ generators from $\mathscr{G}(I)$ satisfying the following conditions :
\\1) $a_{ii}>a_{li}$ for $l=1,\ldots,i-1,i+1,\ldots,t$ and $i=1,\ldots,t$.
\\2) $(a_{11}-1,\ldots,a_{tt}-1)\notin E(I)$.
\begin{theorem}\label{4.1}
    Let $I\subset S$ be an $\mathfrak{m}$-primary monomial ideal. Then
    \begin{center}
          $v(I)=\min\{tr(A)-t~|~A\in\mathcal{M}\}$,
    \end{center}
  where $\mathcal{M}$ is defined as above.
\end{theorem}
\begin{proof}
    Since $I$ is an $\mathfrak{m}$-primary ideal, there exists a monomial $\mathbf{X^b}=\mathbf{X}^{(b_1,\ldots,b_t)}$ such that $(I:\mathbf{X^b})=\mathfrak{m}$. By the definition of colon ideal of monomial ideals, $\mathscr G(I)$ contains at least $t$ elements of the form $g_i=\mathbf{X}^{(b_1-{u_{i1}},b_2-u_{i2},\ldots,b_t-u_{it})}$, where $u_{ij}>0$ for $i\neq j$ and $u_{ii}=-1$ for $i,j\in\{1,2,\ldots,t\}$. 
    Let $B$ be a matrix constructed from $g_1,\ldots,g_t$ such that $B={(E(g_1),\ldots,E(g_t))}^T$. Clearly, $B\in \mathcal{M}$ which implies $\mathcal{M}$ is non-empty.
    Let $N=\{tr(A)-t~|~A\in\mathcal{M}\}$ be a subset of
    $\mathbb{Z}_{\geq 0}$. Suppose $\mathbf{X^c}=\mathbf{X}^{(c_{11}-1,c_{22}-1,\ldots,c_{tt}-1)}$, where $c_{ii}-1\in\mathbb{Z}_{\geq 0}$ for $i=1,\ldots,n$, be a corresponding monomial of $v(I)$ i.e, $(I:\mathbf{X^c})=\mathfrak{m}$, which gives $v(I)=\displaystyle\sum_{i=1}^{t}c_{ii}-t\in N$.
    If $a\in N$, then by the construction of $N$, we have $a=\displaystyle\sum_{i=1}^{t}a_{ii}-t$, where $a_{ii}$ are the diagonal entries of some matrix $A=(a_{ij})\in \mathcal M$. Then $(I:\mathbf{X}^{(a_{11}-1,\ldots,a_{tt}-1)})=\mathfrak m$. Thus, $v(I)\leq a$, which implies $v(I)\leq \min N$. Therefore, $v(I)=\min\{tr(A)-t~|~A\in\mathcal{M}\}$.
\end{proof}
Consequently, as a Corollary, we recover a part of the result by Ficarra and Sgroi.
\begin{corollary}\label{two}\cite[Theorem 3.7]{antonino}
   Let ${I}=\langle x^{a_0},x^{a_1}y^{b_1},x^{a_2}y^{b_2},\cdot\cdot\cdot,x^{a_{n-1}}y^{b_{n-1}},y^{b_n}\rangle\subset K[x,y]$ be a $\mathfrak{m}$-primary monomial ideal with lexicographic ordering with $a_0>a_1>a_2>\ldots>a_{n-1}$ and $b_1<b_2<\ldots<b_n$, i.e. $(a_0,0)>(a_1,b_1)>(a_2,b_2)>\cdot\cdot\cdot>(a_{n-1},b_{n-1})>(0,b_n)$. Then $v(I)$ = $\displaystyle\min_{0\leq i\leq {n-1} }\{a_i+b_{i+1}-2\}$.
\end{corollary}
The next Proposition gives an upper bound for the $v$-number of an $\mathfrak{m}$-primary monomial ideal in terms of the degree of its generators.



\begin{proposition} \label{f}

    Let $I\subset S$ be an $\mathfrak{m}$-primary monomial ideal with $x_{i}^{a_i}\in \mathscr G(I)$ for $i=1,\ldots,t$ and $a_i\in \mathbb{Z}_{> 0}$. Then 
    \begin{center}
        $v(I)\leq \displaystyle\sum_{i=1}^{t}a_i-t$.
    \end{center}
    
\end{proposition}
\begin{proof}
By Theorem \ref{4.1}, $v(I)=\displaystyle\sum_{i=1}^{t}(a_{ii}-1)$, where $A=(a_{ij})$ is a corresponding matrix in $\mathcal{M}$ such that $\mathbf{X^{a_i}}\in\mathscr{G}(I)$, where $\mathbf{a_i}=(a_{i1},a_{i2},\ldots,a_{it})$ is the $i$-th row of $A$. 
Since $I$ is $\mathfrak{m}$-primary, therefore, $a_{ii}\leq a_i$ for all $i=1,2,\ldots,t$, which implies  $\sum_{i=1}^{t}a_{ii}\leq \sum_{i=1}^{t}a_i$. Thus, $v(I)\leq \displaystyle\sum_{i=1}^{t}a_i-t$.
\end{proof}
As an immediate Corollary, we give the necessary and sufficient conditions for the equality to hold in the above Proposition.
\begin{corollary}\label{equal}
Let $I\subset S$ be an $\mathfrak{m}$-primary monomial ideal with $x_{i}^{a_i}\in \mathscr{G}(I)$ for $i=1,\ldots,t$ and $a_i\in \mathbb{Z}_{> 0}$. Then $I=\langle{x}_{1}^{a_1},{x}_{2}^{a_2},\ldots,{x}_{t}^{a_t}\rangle$ if and only if $v(I)= \displaystyle\sum_{i=1}^{t}a_i-t.$
\end{corollary}
\begin{proof}
    By Theorem \ref{4.1}, for  $I=\langle{x}_{1}^{a_1},{x}_{2}^{a_2},\ldots,{x}_{t}^{a_t}\rangle$, we have $v(I)=\displaystyle\sum_{i=1}^{t}a_i-t$. For the converse, let $I$ be any monomial with $v(I)=\displaystyle\sum_{i=1}^{t}a_i-t$. Then corresponding monomial of $v(I)$ can be only of the form $\mathbf{X^a}=\mathbf{X}^{(a_1-1,a_2-1,\ldots,a_t-1)}$. Now, if possible, let $\mathbf{X^b}=\mathbf{X}^{(b_1,b_2,\ldots,b_t)}\in \mathscr G(I)$ for $0\leq b_i< a_i$ and $1\leq i\leq t$, 
    then $\mathbf{X^b}|\mathbf{X^a}$, which implies $\mathbf{X^a}\in I$. This contradicts the fact that $(I:\mathbf{X^a})=\mathfrak m$. Thus $I$ contains only pure power of variables. 
\end{proof}
In the next Proposition, we prove that the  $v$-number of $I$ is bounded above by the  $v$-number  of large powers of $I$.
\begin{proposition}
    Let $I\subset S$ be an $\mathfrak{m}$-primary monomial ideal. Then $v(I^{s+1})\geq v(I)$ for $s\gg 0$.
\end{proposition}
\begin{proof}

    Since $I$ is an $\mathfrak m$-primary monomial ideal, therefore $x_i^{a_i}\in \mathscr{G}(I)$, for all $ i=1,\ldots,t$. By Proposition \ref{e}, we have $v(I^{s+1})\geq (s+1)\alpha(I)-1$. Now, there exists some  $s\gg0,$ such that $(s+1)\alpha(I)-1\geq \sum_{i=1}^{t}a_i-t$
    and by Proposition \ref{f}, we get $\sum_{i=1}^{t}a_i-t\geq v(I)$. Thus, $v(I^{s+1})\geq v(I)$ for $s\gg 0$.
\end{proof}
In  Proposition \ref{Saha}, authors have given an upper bound for the $v$-number of monomial ideals in terms of its colon ideal. In the next Proposition, we give the condition under which this upper bound is attained.
\begin{proposition}\label{4.6}

Let $I\subset S$ be a monomial ideal. Then $v(I^{n})=v(I^n:\mathbf{X}^\mathbf{G})+\deg \mathbf{X}^\mathbf{G}$ for $n\geq 1$, where $\mathbf{X}^\mathbf{G}$ is a proper divisor of monomial corresponding to $v(I^n)$.
\end{proposition}
\begin{proof}
Let $\mathbf{X}^{\mathbf{a}}$ be the monomial corresponding  to $v(I^n)$. Then $(I^n:\mathbf{X}^{\mathbf{a}})=\mathfrak{p}$, where $\mathfrak{p}\in \Ass(I)$. Now, if $\mathbf{X^G}$ is a divisor of $\mathbf{X^a}$, then $\mathbf{X}^{\mathbf{a}}=\mathbf{X}^\mathbf{G}\mathbf{X}^\mathbf{H}$ for some monomial $\mathbf{X^H}$. Thus we can write $((I^{n}:\mathbf{X^G}):\mathbf{X^H})=\mathfrak p$, which implies $v(I^n:\mathbf{X}^\mathbf{G})\leq \deg \mathbf{X}^\mathbf{H}$ and $\deg \mathbf{X}^\mathbf{H}= \deg \mathbf{X}^{\mathbf{a}} - \deg \mathbf{X}^\mathbf{G}$. Thus, $v(I^n:\mathbf{X}^\mathbf{G}) + \deg \mathbf{X}^\mathbf{G}\leq v(I^n)$ and by Proposition \ref{Saha}, we have $v(I^n)\leq v(I^n:\mathbf{X}^\mathbf{G}) + \deg \mathbf{X}^\mathbf{G}$. Therefore, $v(I^{n})=v(I^n:\mathbf{X}^\mathbf{G})+\deg \mathbf{X}^\mathbf{G}$.
\end{proof}
In the next Theorem, we show that the $v$-number of powers of a monomial ideal is bounded above by a linear polynomial.
\begin{theorem}\label{main}

    Let $I\subset S$ be a monomial ideal. Then there exists some positive integer $n_0$ such that $v(I^{n+1})\leq n\alpha(I) +d$ for all $n\geq n_0$, where $d$ is some positive integer.
\end{theorem}
\begin{proof}
    Let $f$ be an element in $I$ such that $\alpha(I)=\deg f$. Then for any $n$, $f^{n}\notin I^{n+1}$. By Proposition \ref{Saha}, we have $v(I^{n+1})\leq v(I^{n+1}:f)+\alpha(I)$. Similarly, $v(I^{n+1}:f)\leq v(I^{n+1}:f^2)+2\alpha(I)$ and by proceeding in this manner, we get $v(I^{n+1})\leq v(I^{n+1}:f^n)+ n\alpha(I)$ for all $n\geq 1$. We consider the ascending chain of ideals $I\subseteq (I^2:f)\subseteq (I^3:f^2)\subseteq\ldots\subseteq(I^{n+1}:f^n)\subseteq (I^{n+2}:f^{n+1})\subseteq \ldots  $ in $S$ which stabilizes as $S$ is Noetherian. Thus, there exists some positive integer $n_0$ such that $(I^{n+1}:f^n) =(I^{n+2}:f^{n+1})$ for all $n\geq n_0$. Therefore, $v(I^{n+1})\leq n\alpha(I) +d$, where $d=v(I^{n+1}:f^n)$ for all $n\geq n_0$.
\end{proof}

The following example illustrates the above Theorem and also shows that the above-mentioned bound is better than the bound discussed in \cite[Proposition]{antonino}.
\begin{example}\label{cycle5}
    Let $I$ be the edge ideal of $C_5$, where $V(C_5)=\{a,b,c,d,e\}$ and $E(C_5)=\{ab,bc,cd,de,ea\}$. Then $v(I)=2$. Note that, $(I^{n+1}:(ab)^n)=I+\langle ec\rangle$ and $v(I^{n+1}:(ab)^n)=1$ for all $n\geq1$. Therefore, $v(I^{n+1})\leq v(I^{n+1}
:(ab)^n)+ \deg(ab)^n=1+2n$ for all $n\geq1$. Further by Proposition \ref{e}, we get $2(n+1)-1\leq v(I^{n+1})$ for all $n\geq1$. Thus, $v(I^{n+1})=2n+1$. In this example, $v(I^{n+1})$ becomes a linear polynomial for $n\geq 1$.
\end{example}
\begin{corollary}
   Let $I\subset S$ be a monomial ideal and $f\in I$ be such that $\deg f=\alpha(I)$ with $v(I^{n+1}:f^n)=\alpha(I)-1$ for large $n$. Then for $n\gg0$, $$v(I^{n+1})=(n+1)\alpha(I)-1. $$
\end{corollary}
\begin{proof}
    By using Theorem \ref{main} and applying the given hypothesis, we have $v(I^{n+1})\leq n\alpha(I)+\alpha(I)-1=(n+1)\alpha(I)-1  $ for $n\gg 0$. From Proposition \ref{e}, we obtain $v(I^{n+1})\geq (n+1)\alpha(I)-1$. Therefore, $v(I^{n+1})=(n+1)\alpha(I)-1$ for $n\gg 0$.
\end{proof}
In the following Propositions, we show that the upper bound is attended for some class of $\mathfrak{m}$-primary monomial ideals for all $n\geq 1$.
\begin{proposition}\normalfont\label{4.8}

    Let $I\subset S$ be an $\mathfrak{m}$-primary monomial ideal and $v(I)=\alpha(I)-1$. Then $v(I^{n+1})= v(I)+ n\alpha(I)$ for all $n\geq 1$. 
\end{proposition}
\begin{proof}
    Let $\mathbf{X^a}$ be a monomial corresponding to $v(I)$ and $\mathbf{X^b}\in\mathscr{G}(I)$ such that $\alpha(I)=\deg \mathbf{X^b}$. Then $\deg \mathbf{X^a}=v(I)=\alpha(I)-1$. Clearly, $\mathbf{X}^{\mathbf{a}+n\mathbf{b}}\notin I^{n+1}$
    , otherwise if $\mathbf{X}^{\mathbf{a}+n\mathbf{b}}\in I^{n+1}$ then $\mathbf{X}^{\mathbf{a}+n\mathbf{b}}=\mathbf{X^{m_1}}\mathbf{X^{m_2}}\ldots\mathbf{X^{m_{n+1}}}$ where $\mathbf{X^{m_i}}\in\mathscr G(I)$ for all $1\leq i\leq n+1$. This implies $\deg \mathbf{X}^{\mathbf{a}+n\mathbf{b}}=\deg\mathbf{X^{m_1}}+\deg\mathbf{X^{m_2}}+\ldots+\deg\mathbf{X^{m_{n+1}}}\geq (n+1)\alpha(I)$. Thus $\alpha(I)-1+n\alpha(I)\geq (n+1)\alpha(I)$, which is a contradiction. Therefore,
    $(I^{n+1}:\mathbf{X}^{\mathbf{a}+n\mathbf{b}})\subset S$ for all $n\geq 1.$  Also, for all $n\geq 1,$ we have
    $(I^{n+1}:\mathbf{X}^{\mathbf{a}+n\mathbf{b}})=((I^{n+1}:\mathbf{X}^{n\mathbf{b}}):\mathbf{X^a})\supseteq(I:\mathbf{X^a})= \mathfrak{m}.$  Thus, $(I^{n+1}:\mathbf{X}^{\mathbf{a}+n\mathbf{b}})=\mathfrak{m}$ which implies $v(I^{n+1})\leq v(I)+n\alpha(I)$ for all $n\geq 1.$ By Proposition \ref{e}, we have $(n+1)\alpha(I)-1=v(I)+n\alpha(I)\leq v(I^{n+1})$ for all $n\geq 1.$ Therefore, $v(I^{n+1})=v(I)+n\alpha(I)$ for all $n\geq 1.$
\end{proof}
Next, we give a class of $\mathfrak m$-primary monomial ideals  for which $v(I)=\alpha(I)-1.$

\begin{example}

Let $I=\langle x^c,y^c\rangle +\langle x^{c-i}y^i~|~a\leq i\leq b\rangle\subset K[x,y]$ where $0<a< b<c$ be an $\mathfrak m$-primary monomial ideal. Clearly, $x^{c-a}y^a ~\mbox{and}~x^{c-(a+1)}y^{a+1}\in \mathscr{G}(I)$ and $\alpha(I)=c$. Then by Corollary \ref{two}, $v(I)\leq (c-a)+(a+1)-2=c-1$ which implies $v(I)=\alpha(I)-1$. Further, by Proposition \ref{4.8}, we have $v(I^{n+1})=v(I)+n\alpha(I)$ for all $n\geq1$. 
\end{example}
   
Now, we prove that for an $\mathfrak m$-primary monomial ideal which is generated by only pure powers of the variables, the upper bound is achieved. The following Lemma will be useful in proving the result.
\begin{lemma}
  \label{6.14}  

    Let $I=\langle {x}_{1}^{a_1},{x}_{2}^{a_2},\ldots,{x}_{t}^{a_t}\rangle$ be an $\mathfrak m$-primary ideal. Then for any $x_{i}^{a_i}\in \mathscr{G}(I)$ and for all $n \geq 1$
    \begin{center}
        $(I^{n+1}:x_{i}^{a_i})=I^{n}.$
    \end{center}
\end{lemma}
\begin{proof}
     Let $g\in \mathscr{G}(I^{n+1}:x_{i}^{a_i})$. Then by the definition of colon ideal, we have $g=\frac{u}{gcd(u,x_{i}^{a_i})}$, where $u\in \mathscr{G}(I^{n+1})$. Note that $gcd(u,x_{i}^{a_i})=x_{i}^{a_i-v}$ for $0\leq v\leq a_{i}$. Since $u\in \mathscr{G}(I^{n+1})$, therefore $u=f_1f_2\ldots f_{n+1}$, where $f_j\in \mathscr{G}(I)$ for $1\leq j\leq n+1$. As $\mathscr{G}(I)$ contains only the pure powers of the variables,  there exists an $f_j\in \mathscr G(I)$ such that $x_{i}^{a_i-v}|f_j$. Thus, $g\in I^{n}$ which implies $(I^{n+1}:x_{i}^{a_i})\subseteq I^{n}$. Clearly, $I^{n}\subseteq (I^{n+1}:x_{i}^{a_i})$. Therefore $(I^{n+1}:x_{i}^{a_i})=I^{n}$ for any $x_{i}^{a_i}\in \mathscr{G}(I)$ and $n \geq 1$.
\end{proof}
\begin{proposition}\normalfont\label{4.10}

    Let $I=\langle {x}_{1}^{a_1},{x}_{2}^{a_2},\ldots,{x}_{t}^{a_t}\rangle$ be an $\mathfrak m$-primary monomial ideal. Then $v(I^{n+1})=v(I)+n\alpha(I)$ for all $n\geq 1$.
\end{proposition}
\begin{proof}
Let $f \in \mathscr{G}(I)$ be such that $\deg f=\alpha(I)$, then by Lemma \ref{6.14}, we get $(I^{n+1}:f^{n})=I$ for all $n\geq1$. Therefore by using Theorem \ref{main}, $v(I^{n+1})\leq v(I)+n\deg f=v(I)+n\alpha(I)$ for all $n\geq1$.
    Next, for a fixed $n\geq 1$, let $v(I^{n+1})= r,$ where $r=\deg h$ and $h$ is a monomial corresponding to $v(I^{n+1})$. Note that since $h\in I^n\backslash I^{n+1}$, then $h=h'f_1\ldots f_{n}$ where $f_i\in \mathscr G(I)$ for $1\leq i\leq n$ and $h'\notin I$. We have
    \begin{equation*} 
\begin{split}
\mathfrak m & = (I^{n+1}:h) \\
 & = ((I^{n+1}:f_1\ldots f_{n}):h')\\
 & = (I:h')~\mbox{(By Lemma \ref{6.14})}.
\end{split}
\end{equation*}

    This implies $v(I)\leq \deg h'$ and $\deg h'= \deg h -\displaystyle\sum_{i=1}^{n}\deg f_i$. Therefore, $v(I)+ \displaystyle\sum_{i=1}^{n}\deg f_i\leq \deg h$, which implies $v(I)+n\alpha(I)\leq v(I^{n+1})$. Thus, $v(I^{n+1})=v(I)+n\alpha(I)$ for all $n\geq 1$.
\end{proof}
\begin{remark}\normalfont
The above statement need not true be for any $\mathfrak m$-primary ideal. Let $I=\langle x^{10},y^{11},z^{12},xy^4z,xy^2z^3,x^3yz^5\rangle\subset K[x,y,z]$ be an $\mathfrak m$-primary ideal. Here $v(I)=14$ and $\alpha(I)=6$. Note that $v(I^2)=17\neq v(I)+\alpha(I)$.
\end{remark}

The next Proposition gives us a bound for the $v$-number of powers of edge ideal of a graph.

\begin{proposition}\label{f1}
    Let $I$ be the edge ideal of a graph $G$. Then $v(I^{n+1})\leq 2n+v(I)$ for all $n\geq 1$.
\end{proposition}
\begin{proof}
    From \cite[Lemma 2.12]{susan}, we have for all $n \geq 1$, $(I^{n+1}:I)=I^n$ , which implies $\depth G(I)\geq 1$. Then $f=\sum_{e_i\in E(G)}e_i\in I\backslash I^2$ is a superficial element of $I$ such that  $(I^{n+1}:f)=I^n$ for all $n\geq 1$ and $\deg f=\alpha(I)$. Therefore, by using Theorem \ref{main} $v(I^{n+1})\leq v(I)+n\alpha(I)$ for all $n\geq 1$. Since $\alpha(I)=2$, $v(I^{n+1})\leq 2n+v(I)$ for all $n\geq 1$.
\end{proof}

In the next Corollary, we observe that if $v(I)=1$, then $v(I^{n+1})$ is a linear function for all $n\geq 1$.
\begin{corollary}\label{for}
    Let $I$ be the edge ideal of a graph $G$ with $v(I)=1$. Then $v(I^{n+1})=2n+1$ for all $n \geq 1$.  
\end{corollary}
\begin{proof}
   By \cite[Proposition 4.3]{antonino}, we have $v(I^{n+1})\geq 2(n+1)-1=2n+1$ for all $n \geq 1$. Since $v(I)=1$, therefore by Proposition \ref{f1}, we get $v(I^{n+1})\leq 2n+1$ for all $n \geq 1$. Thus, $v(I^{n+1})=2n+1$ for all $n \geq 1$.
   \end{proof}
   \begin{remark}
Let $G$ be a graph such that its edge ideal $I(G)$ has linear resolution. Then from \cite[Corollary 4.6]{antonino}, we have $v(I(G))=1$. Therefore by using Corollary \ref{for}, $v(I(G)^{n+1})=2n+1$ for all $n\geq 1$, which gives an alternate proof of \cite[Theorem 4.4]{antonino}.
   \end{remark}
   Next we give one example of a graph whose edge ideal doesn't have linear resolution but $v(I^{n+1})=2n+1$ for $n\geq 1$.
   \begin{example}
 Let us  consider the path $P_5$, then $I=I(P_5)$ does not have a linear resolution. But $v(P_5)=1$. Hence by Corollary \ref{for}, we have $v(I^{n+1})=2n+1$ for all $n\geq 1$.
   \end{example}
In the next Theorem, we compare the $v$-number of an $\mathfrak m$-primary monomial ideal and its regularity.
\begin{theorem}
    Let $I\subset S$ be an $\mathfrak{m}$-primary monomial ideal. Then $v(I)\leq \reg(S/I)$.
\end{theorem}
\begin{proof}
    Since $I$ is $\mathfrak m$-primary monomial ideal, we have $\dim(S/I)=0$. Therefore, $\reg(S/I)
  =a_0(S/I)  =\sup \{n\in\mathbb Z~|~(S/I)_n\neq 0\}$. Let $\mathbf{X^b}=x_{1}^{b_1}x_{2}^{b_2}\ldots x_{t}^{b_t}\in S$ be a monomial of highest degree such that $\overline{\mathbf{X^b}}\neq\overline{0}$ in $S/I$, i.e, $\deg \mathbf{X^b}=\displaystyle\sum_{i=1}^{t}b_i=\reg(S/I)$. Now $\mathbf{X^b}\notin I$, but $h_j=x_j\mathbf{X^b}\in I$ for $1\leq j\leq t$. Then there exists $g_j\in \mathscr G(I)$ such that $g_j|x_j\mathbf{X^b}$. Since $g_j\nmid \mathbf{X^b}$, $g_j=x_jY_j$ where $Y_j|\mathbf{X^b}$ for $1\leq j\leq t$. Therefore $(I:\mathbf{X^b})=\mathfrak m$ and it implies that $v(I)\leq \deg \mathbf{X^b}=\displaystyle\sum_{i=1}^{t}b_i$. This concludes that $v(I)\leq \reg(S/I)$.
\end{proof}
We conclude this section with the following Proposition, which shows that there exists an $\mathfrak{m}$-primary monomial ideal $I$ for which the 
difference between its regularity and $v$-number can be arbitrarily large. 
\begin{proposition}
Let $n$ be any positive integer. Then there exists an $\mathfrak{m}$-primary monomial ideal $I\subset S$ such that $\reg(S/I)-v(I)=n$.
\end{proposition}
\begin{proof}

For any positive integer $n$, let $I=\langle {x}_{i}^{a_i}~|~1\leq i\leq t\rangle +\langle x_{1}^{a_1-u}x_{2}^{a_2-(u+n)}\rangle\subset S$ be an $\mathfrak m$-primary monomial ideal, where $a_1-u,a_2-(u+n)> 0$. Note that by Theorem \ref{4.1}, we get $v(I)=\displaystyle\sum_{i=1}^{t}a_{i}-(u+n+t)$ and $\reg(S/I)=\sup \{i\in\mathbb Z~|~(S/I)_i\neq 0\}=\displaystyle\sum_{i=1}^{t}a_{i}-(u+t)$. Thus, $\reg (S/I)-v(I)=n$.
\end{proof}
\section{Acknowledgements}
We would like to thank Dr. Kamalesh Saha for his valuable suggestions.

\end{document}